\documentclass{elsarticle}
\usepackage{amssymb, latexsym, comment, amsmath, amsthm, upgreek, epsf, epsfig, color, natbib}

\makeatletter
\def\ps@pprintTitle{%
\let\@oddhead\@empty
\let\@evenhead\@empty
\def\@oddfoot{\centerline{\thepage}}%
\let\@evenfoot\@oddfoot}
\makeatother

\newtheorem{thm}{Theorem}[section]
\newtheorem{lemma}[thm]{Lemma}
\newtheorem{prop}[thm]{Proposition}
\newtheorem{cor}[thm]{Corollary}
\newtheorem{definition}[thm]{Definition}
\newtheorem{remark}[thm]{Remark}
\newtheorem{ex}[thm]{Example}

\begin{document}
\begin{frontmatter}
\title{Minimizing Closed Geodesics via Critical Points of the Uniform Energy}
\author{Ian Adelstein \corref{cor1}}
\address{Department of Mathematics, Dartmouth College, Hanover, NH 03755 United States}
\cortext[cor1]{Email: adelstein@dartmouth.edu}
\begin{abstract} In this paper we study 1/k-geodesics, those closed geodesics that minimize on any subinterval of length $l(\gamma)/k$. We employ energy methods to provide a relationship between the 1/k-geodesics and what we define as the balanced points of the uniform energy. We show that classes of balanced points of the uniform energy persist under the Gromov-Hausdorff convergence of Riemannian manifolds. Additionally, we relate half-geodesics (1/2-geodesics) to the Grove-Shiohama critical points of the distance function. This relationship affords us the ability to study the behavior of half-geodesics via the well developed field of critical point theory. Along the way we provide a complete characterization of the differentiability of the Riemannian distance function.
\end{abstract}
\end{frontmatter}

%
%
%
\section{Introduction}

A defining property of a geodesic is that it is a locally distance minimizing curve. It is clear that a nontrivial closed geodesic can never be a \emph{globally} distance minimizing curve. Indeed, a closed geodesic can never minimize past half its length: traversing the geodesic in the opposite direction will always provide a shorter path. It is therefore natural to consider the largest interval on which a given closed geodesic is distance minimizing. This led Sormani \cite{Sor} to consider the notion of a 1/k-geodesic.

\begin{definition} A \emph{1/k-geodesic} is a closed geodesic $\gamma \colon S^1 \to M$ which is minimizing on all subintervals of length $l(\gamma)/k$, i.e. $$d(\gamma(t),\gamma(t+2\pi/k))=l(\gamma)/k \hspace{4mm} \forall t \in S^1.$$ 
\end{definition}

In order to study the existence of 1/k-geodesics on compact Riemannian manifolds Sormani employed the uniform energy $E \colon M^k \to \mathbb{R}$, a function first introduced in the study of Morse theory (Cf.~\cite{Mil}).

\begin{definition}\label{ue} The \emph{uniform energy} $E \colon M^k \to \mathbb{R}$ is defined for an element $\bar{x} = (x_1, \ldots , x_k)  \in M^k $ as 
\begin{equation*}E(\bar{x}) = \sum_{i=1}^k \frac{d(x_i , x_{i+1})^2}{1/k} \hspace{9pt} where \hspace{6pt} x_{k+1} = x_1
\end{equation*}
\end{definition}

Sormani focused on the \emph{openly} 1/k-geodesics: those 1/k-geodesics that do not contain cut points at distance $l(\gamma)/k$ and therefore always minimize on some open neighborhood of subintervals of length $l(\gamma)/k$. She established a one-to-one correspondence between the openly 1/k-geodesics on a compact Riemannian manifold and the smooth critical points of the uniform energy \cite[Theorem 10.2]{Sor}. She then asked if her ideas could be extended to achieve a similar relationship between the full set of 1/k-geodesics on a compact Riemannian manifold and some notion of non-smooth critical points of the uniform energy \cite[Remark 10.4]{Sor}. We provide a positive answer to her problem by introducing the balanced points of the uniform energy.

\begin{definition}\label{balanced_intro} 
Let $\widehat{qp}$ denote the collection of initial velocity vectors in $T_qM$ of unit-speed minimizing geodesics joining $q$ to $p$. Let $\bar{x} = (x_1, \ldots , x_k)  \in M^k$. We say $\bar{x}$ is a \emph{balanced point} of the uniform energy if $d(x_{i-1},x_i) = d(x_i, x_{i+1})$ and if there exists a vector $\bar{\xi} = (\xi_1, \ldots \xi_k) \in T_{\bar{x}}M^k $ with $\xi_i \in \widehat{x_i x_{i+1}}$ and $-\xi_i \in \widehat{x_i x_{i-1} }$ for every $i =1, \ldots, k$.
\end{definition}

\begin{thm}\label{combined} The 1/k-geodesics in a compact Riemannian manifold have a one-to-one correspondence with classes of balanced points of the uniform energy in $M^k$ of nonzero energy.
\end{thm}

Note that closed geodesics may disappear under the Gromov-Hausdorff convergence of Riemannian manifolds (see Example~\ref{sphere2}). In contrast, Sormani demonstrated the persistence of 1/k-geodesics under Gromov-Hausdorff convergence (Theorem~\ref{conv2}).  Classes of smooth critical points of the uniform energy may disappear even under the smooth convergence of Riemannian manifolds (see Example~\ref{sphere}). We combine Sormani's Theorem~\ref{conv2} with our Theorem~\ref{combined} to show that classes of balanced points of the uniform energy persist under Gromov-Hausdorff convergence.

\begin{cor}\label{persist} Let $M_i$ be a sequence of compact Riemannian manifolds converging to a compact Riemannian manifold $M$ in the Gromov-Hausdorff sense. Then any sequence of classes of balanced points of the uniform energies $E \colon M^k_i \to \mathbb{R}$ admits a subsequence which converges to a class of balanced points of the uniform energy $E \colon M^k \to \mathbb{R}$ or to the trivial class. 
\end{cor}

We next turn our attention to the half-geodesics: those closed geodesics that minimize on any subinterval of length $l(\gamma)/2$. We first note that every closed manifold with non-trivial fundamental group admits a half-geodesic. Indeed, the shortest non-contractible closed geodesic, or the systole of the manifold, will always be a half-geodesic \cite[Lemma 4.1]{Sor}. Using Clairaut's relation, Wing Kai Ho \cite{Ho} produced surfaces diffeomorphic to the $2$-sphere that do not admit half-geodesics. In \cite{Ade}, the author produced surfaces diffeomorphic to the $2$-sphere that admit exactly $n$ half-geodesics for every $n\geq 0$.

We provide a relationship between the half-geodesics (1/2-geodesics) and the Grove-Shiohama critical points of the distance function (Definition~\ref{GSdef2} and Proposition~\ref{GS2}). This notion of a critical point for the distance function was introduced by Grove and Shiohama \cite{Gro} during their proof of the celebrated diameter sphere theorem (Theorem~\ref{GSthm2}). Further work on the Grove-Shiohama notion of critical points was conducted by Grove-Petersen \cite{GP}, Gromov \cite{Gromov}, and Abresch-Gromoll \cite{AG}. We employ techniques from this field to study the behavior of the half-geodesics and establish the following result.

\begin{thm}\label{behavior} Let $M$ be a complete Riemannian manifold with sectional curvature $K \geq H>0$. Assume $M$ admits a half-geodesic $\gamma \colon S^1 \to M$ with $l(\gamma) > \frac{\pi}{\sqrt{H}}$. Then any closed geodesic which intersects $\gamma$ must have length at least $l(\gamma)$. Any half-geodesic which intersects at $\gamma(t)$ must also intersect at $\gamma(t+\pi)$ and have length equal to $l(\gamma)$.
\end{thm}

The paper proceeds as follows. In Section~\ref{sec_diff} we give the proof of a folk theorem that completely characterizes the differentiability of the distance function (Theorem~\ref{diff_cut_locus}). This result is applied in Section~\ref{critical} to study  the differentiability of the uniform energy. We are able to strengthen the standard relationship from Morse theory between the critical points of the uniform energy and the derivatives of its component distance functions (Theorem~\ref{sormani}). In Section~\ref{geo} we introduce the balanced points of the uniform energy and provide the proof of Theorem~\ref{combined}. In Section~\ref{balance5} we examine the behavior of the uniform energy functional under Gromov-Hausdorff convergence and prove Corollary~\ref{persist}. In Section~\ref{half2} we detail the relationship between half-geodesics and the Grove-Shiohama critical points of the distance function. It is here that we prove Theorem~\ref{behavior} and show how it applies to the round sphere, to oblate ellipsoids, and to real projective space (Examples~\ref{round}, \ref{oblate}, \ref{proj}).

%
%
%
%
\section{Differentiability of the Distance Function}\label{sec_diff}

Let $M$ be a complete Riemannian manifold. It is well known that the distance function $d_p \colon M \to \mathbb{R}$ is a continuous function on $M$ and is of class $C^{\infty}$ (i.e.~smooth) away from the cut locus of $p$ and the point $p$ itself (Cf.~\cite{Sak}). Less well known, and apparently absent from the literature, is the fact that $d_p$ is indeed differentiable at a certain subset of the cut locus of $p$. The main result of this section is Theorem~\ref{diff_cut_locus} which provides a complete characterization of the differentiability of the distance function $d_p$. In addition to filling a gap in the literature, this result will allow us to explore the differentiability of the uniform energy in the following section.

Let $M$ be a complete Riemannian manifold and $p \in M$ with $\gamma \colon [0,\infty) \to M$ a normalized geodesic, $\gamma(0)=p$. By continuity, the set of numbers $t >0$ such that $t=d(p,\gamma(t))$ is of the form $[0, t_0]$ or $[0, \infty)$. In the first case we say $\gamma(t_0)$ is the \emph{cut point} of $p$ along $\gamma$, and in the second case we say such a cut point does not exist. We define the \emph{cut locus} of p, denoted $C(p)$, to be the collection of all such cut points. A basic fact about the cut locus is as follows.

\begin{prop}[Cf.~\cite{doC}] Let $\gamma(t_0)$ be the cut point of $p=\gamma(0)$ along $\gamma$. Then
\begin{description}
\item{a)} either $\gamma(t_0)$ is the first conjugate point of $p$ along $\gamma$,
\item{b)} or there exists a geodesic $\sigma \neq \gamma$ from $p$ to $\gamma(t_0)$ such that $l(\sigma)=l(\gamma)$.
\item{} Conversely if (a) or (b) is satisfied, then there exists $\tilde{t} \in (0,t_0]$ such that $\gamma(\tilde{t})$ is the cut point of p along $\gamma$.
\end{description}
\end{prop}

We partition the cut locus into two disjoint subsets. We say that a cut point $q$ of $p$ is \emph{singular} if there exists a unique minimizing geodesic $\gamma$ joining $q$ to $p$. This can only occur if the cut points $p$ and $q$ are conjugate along $\gamma$. However being conjugate does not guarantee the existence of a unique minimal geodesic. We say that a cut point $q$ of $p$ is \emph{ordinary} if there exist multiple minimizing geodesics joining $q$ to $p$. Bishop \cite{Bis} showed that the set of ordinary cut points is dense in the cut locus. We are now ready to give a complete characterization of the differentiability of the distance function.

\begin{thm}\label{diff_cut_locus}Let $d_p \colon M \to \mathbb{R}$ denote the distance function from $p$. Then
\begin{description}
\item{a)} if $q \not\in C(p) \cup \{p\}$, we have that $d_p$ is smooth at $q$ and $\nabla d_p (q) = \dot{\sigma} (l)$ where $\sigma$ is the unique minimizing geodesic joining $p$ to $q$ and $l=d_p(q)$.
\item{b)} if $q \in C(p)$ is an ordinary cut point of $p$, we have that $d_p$ is not differentiable at $q$.
\item{c)} if $q \in C(p)$ is a singular cut point of $p$, we have that $d_p$ is differentiable at $q$ and $\nabla d_p (q) = \dot{\sigma} (l)$ where $\sigma$ is the unique minimizing geodesic joining $p$ to $q$ and $l=d_p(q)$.  However $d_p$ is not $C^1$ at its singular cut points.
\end{description}
\end{thm}

Sakai \cite{Sak} provides a detailed proof of parts (a) and (b) of Theorem~\ref{diff_cut_locus}. Ivanov \cite{Iva2} makes a statement similar to part (c) in a MathOverflow response. He does not provide a proof and it appears as though one does not exist in the literature. We provide one here for completeness. 

\begin{definition} \label{d_p_definition} The \emph{one-sided directional derivative} of the distance function at the point $q \in M$ in the direction $v \in T_qM$ is defined to be
\begin{equation*}\label{d_p_def} D^+_v d_p(q) = \lim_{t \to 0^+} \frac{d_p(\exp_q (t \cdot v)) - d_p(q)}{t}
\end{equation*}
provided that this limit exists. 
\end{definition}

\begin{lemma}[Cf.~\cite{Ito, Pla}]\label{direction_deriv} Let $\widehat{qp}$ denote the collection of initial velocity vectors in $T_qM$ of unit-speed minimizing geodesics joining $q$ to $p$. Then for a vector $v \in T_qM$ the one-sided derivative in the direction $v$ exists and is given by
\begin{equation*}\label{direc_deriv} D^+_v d_p (q) = \min{ \{ - \langle v,\xi  \rangle   :  \xi \in \widehat{qp}   \}. } 
\end{equation*}
In particular, if there exists a unique minimizing geodesic $\sigma$ joining $p$ to $q$ then the two-sided directional derivative exists and equals $\langle v, \dot{\sigma}(l) \rangle$ where $l=d_p(q)$.
\end{lemma}

\begin{proof}[Proof of Theorem~\ref{diff_cut_locus}(c).]  

Let $q$ be a singular cut point of $p$. Let $\sigma$ be the unique minimizing geodesic joining $p$ to $q$ and $l=d_p(q)$. It suffices to show that for every sequence $h_i \in T_qM$ with $|h_i| \to 0$, there exists a subsequence $k_i$ such that 
\begin{equation} \label{lim}
\lim_{i \to \infty} \frac{d_p(\exp_q(k_i)) - d_p(q) - \langle \dot{\sigma}(l) , k_i \rangle}{|k_i|} =0.
\end{equation}
Write $h_i = t_i u_i$ for $u_i \in U_qM$ the unit tangent sphere and $t_i \to 0$. By compactness of the unit sphere we can choose a subsequence $v_i$ of the $u_i$ which converge to some $v \in U_qM$. Setting $k_i = t_i v_i$ we will show that the limit in (\ref{lim}) holds.

From Lemma~\ref{direction_deriv} we have that $$\lim_{i \to \infty} \frac{d_p(\exp_q(t_i v)) - d_p(q) - \langle \dot{\sigma}(l) , t_i v \rangle}{t_i} =0.$$ 
Letting $\gamma_i=\exp_q(t_i v)$ and $c_i = \exp_q(t_i v_i)$ and applying the triangle inequality twice we have $$d_p(\gamma_i)-d(c_i,\gamma_i) \leq d_p(c_i) \leq d_p(\gamma_i)+d(c_i,\gamma_i). $$ 
We note that $\lim_{i \to \infty} \frac {\langle \dot{\sigma}(l) , t_i v_i \rangle}{t_i} = \lim_{i \to \infty} \frac {\langle \dot{\sigma}(l) , t_i v \rangle}{t_i}$ and $\lim_{i \to \infty} \frac{|t_i v - t_i v_i |}{t_i}=0$ so that the limit in (\ref{lim}) will hold if we show that $$\lim_{i \to \infty} \frac{d(c_i , \gamma_i )}{t_i} \leq \lim_{i \to \infty} \frac{|t_i v - t_i v_i | + O(t_i^2)}{t_i} =0.$$

Let $\tilde{\sigma}_{t_i}(s) \colon [0,1] \to T_qM$ be the constant speed path in $T_qM$ which traverses the arc of the circle of radius $t_i$ between $t_i v$ and $t_i v_i$ and $\sigma_{t_i}(s) \colon [0,1] \to M$ the curve in $M$ given by $\sigma_{t_i}(s)= \exp_q\,\tilde{\sigma}_{t_i}(s)$ so that $\sigma_{t_i}(0)=\gamma_i$ and $\sigma_{t_i}(1)=c_i$. We choose $i$ large enough that we can work in normal coordinates around $q$ where we know that $g_{ij} (\sigma_{t_i}(s))= \delta_{ij} + O(t_i^2)$ so that $|\dot{\sigma}_{t_i}(s)|^2= |\dot{\tilde{\sigma}}_{t_i}(s)|^2+O(t_i^2).$ Now computing the length of $\sigma_{t_i}(s)$ we have
\begin{align*}
l(\sigma_{t_i}(s) ) &= \int_0^1 |\dot{\sigma}_{t_i}(s)| \, ds = \int_0^1 \sqrt{  |\dot{\tilde{\sigma}}_{t_i}(s)|^2 + O(t_i^2)}\, ds =\int_0^1  |\dot{\tilde{\sigma}}_{t_i}(s)| + O(t_i^2) \, ds
\\ &= l(\tilde{\sigma}_{t_i}(s) ) + O(t_i^2) =   |t_i v - t_i v_i | +O(t_i^2)
\end{align*}
Finally we note that $d(c_i, \gamma_i ) \leq l(\sigma_{t_i}(s)) = |t_i v - t_i v_i |   +O(t_i^2) $ so that indeed the limit in (\ref{lim}) holds and we have shown that $d_p$ is differentiable at $q$ with $\nabla d_p (q)=\dot{\sigma}(l)$.

It is left to show that $d_p$ is not $C^1$ at singular cut points. Bishop \cite{Bis} showed that the ordinary cut locus is dense in the cut locus. Therefore, there exists a sequence $\{q_n\}$ of ordinary cut points converging to the singular cut point $q$. We know from part (b) of the theorem that $\nabla d_p$ is not defined at $q_n$ and can therefore conclude that $d_p$ is not $C^1$ at the singular cut locus.
\end{proof}

%
%
%
%
%
\section{Critical Points of the Uniform Energy}\label{critical}

The uniform energy (Definition~\ref{ue}) is given as a sum of distance functions, hence its differentiability will depend on the differentiability of its component distance functions. A priori, it may be possible for $E \colon M^k \to \mathbb{R}$ to be differentiable at points $\bar{x} \in M^k$ for which the individual terms $d(x_i,x_{i+1})^2$ are not differentiable. In Lemma~\ref{no_ordinary} we show that this phenomenon does not occur at the critical points of the uniform energy (Definition~\ref{cp}). We apply this lemma in Theorem~\ref{sormani} to strengthen the standard result from Morse theory relating the critical points of the uniform energy and the derivatives of its component distance functions (Cf.~\cite{Mil}). In particular, we are able to remove the standard assumption that the distance between each pair $(x_i,x_{i+1})$ is less than the injectivity radius of $M$.

\begin{definition}\label{cp} Let $\bar{x} = (x_1, \ldots , x_k)  \in M^k $. We say $\bar{x}$ is a \emph{critical point} of the uniform energy if $E$ is differentiable at $\bar{x}$ and its gradient is zero, i.e.~$\nabla E(\bar{x})=\vec{0}$. 
\end{definition}

\begin{lemma}\label{no_ordinary} If $\bar{x} \in M^k$ is a critical point of the uniform energy then none of the pairs $(x_i, x_{i+1})$ are ordinary cut points, i.e.~all of the component distance functions are differentiable.
\end{lemma}

\begin{proof} We know that $\nabla E(\bar{x})=\vec{0}$ because $\bar{x}$ is a critical point. Letting $\bar{v}_i=(0, \ldots, v_i, \ldots, 0) \in T_{\bar{x}}M^k$ we have that $D_{\bar{v}_i} E (\bar{x}) =  \langle \bar{v}_i, \nabla E (\bar{x})  \rangle =0$ for every $\bar{v}_i \in T_{\bar{x}}M^k$. We show how to calculate $D_{\bar{v}_i} E (\bar{x})$ using the one-sided directional derivative of the distance function from Lemma~\ref{direction_deriv}. We have
\begin{align*}
D_{\bar{v}_i} E (\bar{x}) &=  k \, D^+_{v_i} (d_{x_{i-1}}^2 + d_{x_{i+1}}^2 )x_i \\
&=  2k \, d(x_{i-1},x_i) \, D^+_{v_i} (d_{x_{i-1}} )x_i + 2k \, d(x_i , x_{i+1}) \, D^+_{v_i} (d_{x_{i+1}} )x_i \\
&= c_1 \, \min{ \{ - \langle v_i,\xi  \rangle   :  \xi \in \widehat{x_i x_{i-1}} } \} + c_2 \, \min{ \{ - \langle v_i,\eta  \rangle   :  \eta \in \widehat{x_i x_{i+1}} } \}
\end{align*}
where $c_1=2k \, d(x_{i-1},x_i) $ and $c_2= 2k \, d(x_i, x_{i+1})$ are both positive constants.

We will show that the sets $\widehat{x_i x_{i-1}}$ and $\widehat{x_i x_{i+1}}$ each consist of a single vector so that by Theorem~\ref{diff_cut_locus} we can conclude that each of the component distance functions is differentiable. Assume by contradiction that one (or both) of the sets contains multiple vectors; say there exist $\xi_1$ and $\xi_2 \in \widehat{x_i x_{i-1}}$ with $\xi_1 \neq \xi_2$. Then for any $\eta \in \widehat{x_i x_{i+1}}$ we must have that either $\xi_1 \neq - \eta$ or  $\xi_2 \neq - \eta$ because $\xi_1 \neq \xi_2$. Without loss of generality assume that $\xi_1 \neq - \eta$. Then there must exist  $v_i \in T_{x_i}M$ such that $- \langle v_i, \xi_1  \rangle < 0$ and $- \langle v_i, \eta \rangle < 0$. This yields $D_{\bar{v}_i} E(\bar{x}) < 0$ which is a contradiction when $\bar{x}$ is a critical point. We therefore conclude that none of the pairs $(x_i, x_{i+1})$ are ordinary cut points and that all of the component distance functions are differentiable.
\end{proof}

\begin{thm}\label{sormani} We have that $\bar{x}$ is a critical point of the uniform energy if and only if $ d(x_{i-1},x_i) = d(x_i, x_{i+1})$ and $\nabla d_{x_{i-1}} x_i = -\nabla d_{x_{i+1}} x_i$ for every $i \in \{1, \ldots, k \}$.
\end{thm}
\begin{proof} Let $\bar{x}$ be a critical point so that by Lemma~\ref{no_ordinary} none of the pairs $(x_i, x_{i+1})$ are ordinary cut points and $\nabla d_{x_{i \pm 1}} x_i $ is well defined for every $i$. Computing via the chain rule we have 
\begin{align}\label{chain}
\nabla E(\bar{x}) = 2k \, &[d_{x_k}x_1 \, \nabla d_{x_k}x_1 +d_{x_2}x_1 \, \nabla d_{x_2}x_1 , \, d_{x_1}x_2 \, \nabla d_{x_1}x_2 +d_{x_3}x_2 \, \nabla d_{x_3}x_2, \notag  \\
&  \ldots ,d_{x_{k-1}}x_k \, \nabla d_{x_{k-1}}x_k +d_{x_1}x_k \, \nabla d_{x_1}x_k ]
\end{align}
so that the critical point condition $\nabla E(\bar{x}) = \vec{0}$ implies that $ d_{x_{i-1}}x_i \, \nabla d_{x_{i-1}} x_i = -d_{x_{i+1}}x_i \, \nabla d_{x_{i+1}} x_i$. We note that $|\nabla d_{x_{i \pm 1}} x_i |=1$ and therefore conclude that $ d(x_{i-1},x_i) = d(x_i, x_{i+1})$ and $\nabla d_{x_{i-1}} x_i = -\nabla d_{x_{i+1}} x_i$ for every $i \in \{1, \ldots, k \}$.

In the opposite direction assume $ d(x_{i-1},x_i) = d(x_i, x_{i+1})$ and $\nabla d_{x_{i-1}} x_i = -\nabla d_{x_{i+1}} x_i$ for every $i$. Then by (\ref{chain}) we have $\nabla E(\bar{x}) = \vec{0}$ so that indeed $\bar{x}$ is a critical point of the uniform energy.
\end{proof}

%
%
%
%
\section{1/k-Geodesics and Balanced Points of the Uniform Energy}\label{geo}

This section introduces the balanced points of the uniform energy and extends the Sormani result \cite[Thm 10.2]{Sor} to provide a one-to-one correspondence between the classes of balanced points of the uniform energy and all 1/k-geodesics on a Riemannian manifold. We begin with an exposition of Sormani's ideas.

\begin{definition}[\cite{Sor}, Definition 3.1]\label{kk} A \emph{1/k-geodesic} is a closed geodesic $\gamma \colon S^1 \to M$ which is minimizing on all subintervals of length $l(\gamma)/k$, i.e. $$d(\gamma(t),\gamma(t+2\pi/k))=l(\gamma)/k \hspace{4mm} \forall t \in S^1.$$ The 1/k-geodesics come in two disjoint flavors:
\begin{description}
\item{a)} A \emph{strict} 1/k-geodesic is a 1/k-geodesic that contains a pair of cut points at distance $l(\gamma)/k$.
\item{b)} An \emph{openly} 1/k-geodesic does not contain cut points at distance $l(\gamma)/k$, and therefore will always minimize on some open neighborhood of subintervals of length $l(\gamma)/k$.
\end{description}
\end{definition}

\begin{prop}[\cite{Sor}, Theorem 3.1]\label{all} Any closed geodesic is a 1/k-geodesic for a sufficiently large number k. 
\end{prop}
\begin{proof} This follows from compactness of the circle and the local minimizing property of geodesics.
\end{proof}

\begin{definition}[\cite{Sor}, Theorem 10.1] Let $\bar{x} \in M^k$ be a critical point of the uniform energy as in Definition~\ref{cp} (i.e.~$\nabla E(\bar{x})=\vec{0}$). We say $\bar{x}$ is a \emph{smooth} critical point of the uniform energy if none of the pairs $(x_i, x_{i+1})$ are (singular) cut points, i.e.~the component distance functions are all smooth functions.
\end{definition}

\begin{lemma}[\cite{Sor}, Lemma 10.2] Let $\bar{x} = (x_1, \ldots , x_k)  \in M^k $ be a smooth critical point of the uniform energy. Then there exists a unique associated closed geodesic $\gamma \colon S^1 \to M$ with $\gamma(2\pi i /k)=x_i$. 
\end{lemma}

\begin{definition}[\cite{Sor}, Definition 10.2] A smooth critical point is called \emph{rotating} if the unique associated geodesic $\gamma \colon S^1 \to M$ is such that for every $t \in S^1$ the point $(\gamma(t), \gamma(t+2\pi/k), \ldots , \gamma(t + 2\pi(k-1)/k) ) \in M^k$ is also a smooth critical point of the uniform energy. In this case we call the set $\{(\gamma(t), \gamma(t+2\pi/k), \ldots , \gamma(t + 2\pi(k-1)/k)) \in M^k  :  t \in S^1 \}$ a \emph{smooth class of critical points} of the uniform energy in $M^k$. 
\end{definition}

A smooth critical point need not be a rotating smooth critical point. Indeed, for $\gamma$ the unique closed geodesic associated to the smooth critical point, there may be a $t_0 \in S^1$ such that $(\gamma(t_0), \gamma(t_0+2\pi/k), \ldots , \gamma(t_0 + 2\pi(k-1)/k) )$ is not a smooth critical point. An example where this occurs is the over-under geodesic on the doubled square \cite[Ex 9.2]{Sor}. Any four-tuple of evenly spaced points on this geodesic is a smooth critical point, except for the four-tuple of midpoints. Both the over and the under geodesics minimize between pairs of adjacent midpoints, hence these pairs are cut points. This geodesic is therefore a strict 1/4-geodesic.

\begin{thm}[\cite{Sor}, Theorem 10.2]\label{sormani3} Openly 1/k-geodesics in a compact Riemannian manifold have a one-to-one correspondence with smooth classes of critical points of the uniform energy in $M^k$ of nonzero energy.
\end{thm}

Sormani then asks if her ideas can be extended to achieve a similar relationship between the strict 1/k-geodesics and some notion of non-smooth critical points of the uniform energy \cite[Remark 10.4]{Sor}. In order to extend her ideas it will be necessary to consider those points $\bar{x} \in M^k$ which contain pairs $(x_i, x_{i+1})$ of cut points. We note that Lemma~\ref{no_ordinary} concludes that the critical points of the uniform energy cannot contain pairs $(x_i, x_{i+1})$ of ordinary cut points. It is therefore necessary to consider a more general notion than critical point of the uniform energy. We introduce the balanced points of the uniform energy for this purpose. The remainder of this section is used to establish Theorem~\ref{combined} which gives a one-to-one correspondence between 1/k-geodesics and classes of balanced points of the uniform energy.

\begin{definition}\label{balanced} Let $\widehat{qp}$ denote the collection of initial velocity vectors (in $T_qM$) of unit-speed minimizing geodesics joining $q$ to $p$.  Let $\bar{x} = (x_1, \ldots , x_k)  \in M^k $
\begin{description}
\item{a)} we say $\bar{x}$ is a \emph{balanced point} of the uniform energy $E \colon M^k \to \mathbb{R}$ if $d(x_{i-1},x_i) = d(x_i, x_{i+1})$ and if there exists a vector $\bar{\xi} = (\xi_1, \ldots \xi_k) \in T_{\bar{x}}M^k $ with $\xi_i \in \widehat{x_i x_{i+1}}$ and $-\xi_i \in \widehat{x_i x_{i-1} }$ for every $i \in \{1, \ldots, k \}$.
\item{b)} such a balanced point is called \emph{uniquely balanced} if none of the pairs $(x_i, x_{i+1})$ are ordinary cut points, i.e. every pair is joined by a unique minimizing geodesic.
\item{c)} such a balanced point is called \emph{smooth} if none of the pairs $(x_i, x_{i+1})$ are cut points and is called \emph{non-smooth} if at least one of the pairs are cut points (ordinary or singular).
\end{description}
\end{definition}

\begin{remark}\normalfont It follows immediately from Theorem~\ref{sormani} that $\bar{x}$ is a critical point of the uniform energy if and only if it is a uniquely balanced point. Similarly, $\bar{x}$ is a smooth critical point of the uniform energy if and only if it is a smooth balanced point.
\end{remark}

\begin{definition} We say a closed geodesic $\gamma \colon S^1 \to M$ is \emph{associated} to the balanced point $\bar{x} \in M^k$ if there exists $t \in S^1$ with $\bar{x}=(\gamma(t), \gamma(t+2\pi/k), \ldots , \gamma(t + 2\pi(k-1)/k) ) \in M^k$ and $\dot{\gamma}(t+2\pi (i-1) /k) \in \widehat{x_i x_{i+1}}$ and $-\dot{\gamma}(t+2\pi (i-1) /k) \in \widehat{x_i x_{i-1}}$ for every $i \in \{1, \ldots, k \}$.
\end{definition} 

We note that a balanced point need not have an associated closed geodesic. An example of such a balanced point is given by the five-tuple of midpoints on the doubled pentagon \cite[Remark 10.4]{Sor}. 

\begin{definition} A balanced point is called \emph{rotating} if there exists an associated closed geodesic $\gamma \colon S^1 \to M$ such that for every $t \in S^1$ the point $(\gamma(t), \gamma(t+2\pi/k), \ldots , \gamma(t + 2\pi(k-1)/k) ) \in M^k$ is a balanced point with associated closed geodesic $\gamma$. In this case we call the set $\{(\gamma(t), \gamma(t+2\pi/k), \ldots , \gamma(t + 2\pi(k-1)/k)) \in M^k  :  t \in S^1 \}$ a \emph{class of balanced points} of the uniform energy in $M^k$. 
\end{definition}

An example of a balanced point with associated closed geodesic which is not rotating is given by the corners on the doubled square \cite[Ex 10.1]{Sor}. 

\begin{definition} A \emph{non-smooth class of balanced points} of the uniform energy in $M^k$ is a class of balanced points such that $(\gamma(t_0), \gamma(t_0 +2\pi/k), \ldots , \gamma(t_0 + 2\pi(k-1)/k) ) \in M^k$ is a non-smooth balanced point for some $t_0 \in S^1$.
\end{definition}

An individual non-smooth balanced point may be associated to multiple closed geodesics, and thus multiple non-smooth classes of balanced points. An example where this occurs is the four-tuple of midpoints on the doubled square which is associated to the two closed over-under geodesics \cite[Ex 9.2]{Sor}. The four-tuple $[(0,0),(.25,.5),(.5,1),(.75,.5)]$ on the standard flat two-torus provides another illustrative example.

\begin{lemma}\label{lem1} For every 1/k-geodesic $\gamma \colon S^1 \to M$ and every $t\in S^1$ the point $\bar{x}=(\gamma(t), \gamma(t+2\pi/k), \ldots , \gamma(t + 2\pi(k-1)/k) ) \in M^k$ is a rotating balanced point of the uniform energy. Therefore, to each 1/k-geodesic we can associate a unique class of balanced points in $M^k$.
\end{lemma}
\begin{proof} We have that $d(x_{i-1},x_i) = d(x_i, x_{i+1})$ because the geodesic $\gamma$ is parametrized by arc length. Then letting $\xi_i$ correspond to the velocity vector of $\gamma$ at $x_i$ we see that $\bar{x}$ is indeed a balanced point of the uniform energy. The rotating condition is satisfied because $\bar{x}$ was generated by the closed 1/k-geodesic $\gamma$.
\end{proof}

\begin{lemma}\label{strict_lemma} For every strict 1/k-geodesic there exists a $t_0 \in S^1$ such that $(\gamma(t_0), \gamma(t_0+2\pi/k), \ldots , \gamma(t_0 + 2\pi(k-1)/k) )$ is a non-smooth balanced point of the uniform energy. Therefore, to each strict 1/k-geodesic we can associate a unique non-smooth class of balanced points in $M^k$.
\end{lemma}
\begin{proof} By Lemma~\ref{lem1} we know that the k-tuple will be a rotating balanced point. The fact that $\gamma$ is a strict 1/k-geodesic means that it contains cut points at distance $l(\gamma)/k$.
\end{proof}

\begin{lemma}\label{relation} To any non-smooth class of balanced points in $M^k$ we can associate a unique strict 1/k-geodesic.
\end{lemma}
\begin{proof} By definition, any class of balanced points in $M^k$ comes with a unique associated 1/k-geodesic. The fact that the class of balanced points is non-smooth ensures that there will be a $t_0 \in S^1$ such that $(\gamma(t_0), \gamma(t_0 +2\pi/k), \ldots , \gamma(t_0 + 2\pi(k-1)/k) ) \in M^k$ is a non-smooth balanced point and therefore that the associated 1/k-geodesic will be strict. 
\end{proof}

\begin{thm}\label{strict} Strict 1/k-geodesics in a compact Riemannian manifold have a one-to-one correspondence with non-smooth classes of balanced points of the uniform energy in $M^k$ of nonzero energy.
\end{thm}
\begin{proof} This fact follows directly from Lemma~\ref{strict_lemma} and Lemma~\ref{relation}.
\end{proof}

This theorem successfully extends the ideas of Sormani to the case where the 1/k-geodesics are not open and the balanced points are non-smooth. We can now combine Theorem~\ref{sormani3} and Theorem~\ref{strict} to obtain Theorem~\ref{combined} which gives a correspondence between classes of balanced points in $M^k$ (smooth or non-smooth) and 1/k-geodesics (open or strict).

%
%
%
\section{Balanced Points Under Gromov-Hausdorff Convergence}\label{balance5}

In this section we study the behavior of the uniform energy functional under Gromov-Hausdorff convergence. In \cite[Example 9.3]{Sor} Sormani showed that classes of smooth critical points of the uniform energy can disappear even under the smooth convergence of Riemannian manifolds (see Example~\ref{sphere}). In contrast, we prove Corollary~\ref{persist} which shows that classes of balanced points of the uniform energy persist under Gromov-Hausdorff convergence. We begin with an important result by Sormani concerning the persistence of 1/k-geodesics under Gromov-Hausdorff convergence.

\begin{thm}[\cite{Sor}, Theorem 7.1]\label{conv2} Let $M_i \to M$ be a sequence of compact Riemannian manifolds converging in the Gromov-Hausdorff sense. Let $\gamma_i \colon S^1 \to M_i$ be a sequence of $1/k$-geodesics. Then a subsequence of the $\gamma_i$ converge point-wise to a continuous curve $\gamma \colon S^1 \to M$, and $\gamma$ is either a $1/k$-geodesic or trivial. 
\end{thm}

\begin{ex}[\cite{Sor}, Example 7.2]\normalfont \label{sphere2} This example illustrates that unlike 1/k-geodesics, closed geodesics may disappear under Gromov-Hausdorff convergence. Consider the sequence of ellipsoids $M_i$ given by $x^2 + y^2 + (z/c_i)^2 =1$. We see that the equators $\gamma_i=(\cos{t},\sin{t},0)$ of these ellipsoids are all closed geodesics. As $c_i \to 0$ the sequence $M_i$ converges in the Gromov-Hausdorff sense to a doubled disk. The equator on the doubled disk is not a locally length minimizing curve, hence not a closed geodesic, and we have demonstrated a disappearing sequence of closed geodesics under Gromov-Hausdorff convergence.

Let us view this example in light of Theorem~\ref{conv2}. We note by Proposition~\ref{all} that each equator $\gamma_i$ is a 1/k-geodesic for some integer $k$. However, as $c_i \to 0$ this integer $k$ grows without bound. The sequence of equators $\gamma_i$ is therefore not a sequence of 1/k-geodesics for any fixed integer $k$, and the example does not contradict the statement of Theorem~\ref{conv2}.
\end{ex}

\begin{ex}[\cite{Sor}, Example 9.3]\normalfont \label{sphere} This example illustrates that classes of smooth critical points of the uniform energy can disappear even under the smooth convergence of Riemannian manifolds. Consider again the ellipsoids given by $x^2 + y^2 + (z/c)^2 =1$. We note that the distance between cut points along the equators $\gamma_c$ is a continuous function of the parameter $c$. There will therefore be some $c_0$ such that $\gamma_{c_0}$ is a strict $1/3$-geodesic. Then for all $c>c_0$ we have that $\gamma_c$ is an openly $1/3$-geodesic and by Theorem~\ref{sormani3} corresponds to a class of smooth critical points of the uniform energy. Choosing a sequence $c_i \to c_0$ with $c_i > c_0$ and considering the associated sequence $M_i$ of ellipsoids converging to $M_{c_0}$ we see that the classes of smooth critical points associated to the equators disappear in the limit space $M_{c_0}$.

We note by Theorem~\ref{combined} that each of the equator $1/3$-geodesics (even that in the limit $M_{c_0}$) is associated to a class of balanced points of the uniform energy, leading us to the statement of Corollary~\ref{persist} that classes of balanced points persist in the limit.
\end{ex}

\begin{proof}[Proof of Corollary~\ref{persist}] Let $M_i$ be a sequence of compact Riemannian manifolds converging to a compact Riemannian manifold $M$ in the Gromov-Hausdorff sense. Assume that each $M_i$ admits a class of balanced points of the uniform energy $E \colon M_i^k \to \mathbb{R}$. Then by Theorem~\ref{combined} we know that each class corresponds to a $1/k$-geodesic $\gamma_i \colon S^1 \to M_i$ and by Theorem~\ref{conv2} there exists a subsequence of the $\gamma_i$ converging to $\gamma\colon S^1 \to M$, where $\gamma$ is either a $1/k$-geodesic or trivial. When $\gamma$ is nontrivial Theorem~\ref{combined} tells us that there exists a corresponding class of balanced points of the uniform energy $E \colon M^k \to \mathbb{R}$.
\end{proof}

We note that individual balanced points of the uniform energy need not persist under Gromov-Hausdorff convergence. In \cite[Example 2.2]{Sor} we see that each of the handled spheres $M_j$ admits a balanced point of the uniform energy $E \colon M^2_j \to \mathbb{R}$ which disappears on the standard sphere in the limit.

%
%
%
%

\section{Half-geodesics and Grove-Shiohama Critical Points of Distance}\label{half2}

The half-geodesics are those closed geodesics which minimize on any subinterval of length $l(\gamma)/2$. It is clear that a half-geodesic will never be an openly half-geodesic and will always be a strict half-geodesic. In Section~\ref{geo} we generalized a result of Sormani to provide a relationship between the strict 1/k-geodesics and the non-smooth balanced points of the uniform energy. We will now apply this result to study the half-geodesics via critical point methods. For an element $\bar{x} = (x_1,x_2)  \in M^2 $ we calculate the uniform energy to be 
\begin{equation*}E(\bar{x}) = \sum_{i=1}^2 \frac{d(x_i , x_{i+1})^2}{1/2} = 4 \cdot d(x_1,x_2)^2
\end{equation*}
We know that the distance function is not everywhere smooth, so we recall the definition of a Grove-Shiohama critical point of the distance function and relate this notion to the balanced points of the uniform energy $E \colon M^2 \to \mathbb{R}$. 

\begin{definition}\label{GSdef2} A \emph{Grove-Shiohama critical point} of $d_p \colon M \to \mathbb{R}$ is a point $q \in M$ such that for any $v \in T_qM$ there exists $\xi \in \widehat{qp}$ such that $\measuredangle {(v, \xi)}  \leq \pi/2$.
\end{definition}

The original application of this critical point definition is the celebrated Grove-Shiohama \cite{Gro} diameter sphere theorem.

\begin{thm}[Grove-Shiohama]\label{GSthm2} Let $M$ be a complete Riemannian manifold with sectional curvature $K \geq H > 0$ and diameter $> \frac{\pi}{2\sqrt{H}}$. Then $M$ is homeomorphic to the sphere. 
\end{thm}

\begin{prop}\label{GS2} If $(p,q) \in M^2$ is a balanced point of the uniform energy $E \colon M^2 \to \mathbb{R}$ then we have that $q$ is a Grove-Shiohama critical point of $d_p$ and that $p$ is a Grove-Shiohama critical point of $d_q$.
\end{prop}
\begin{proof} The balancing condition ensures that there are vectors $\xi_1, \xi_2 \in \widehat{qp}$ which meet at angle $\pi$ in $T_qM$. Thus $q$ is a Grove-Shiohama critical point of $d_p$. The proof that $p$ is a Grove-Shiohama critical point of $d_q$ is equivalent.
\end{proof}

The Grove-Shiohama definition proved fruitful and sparked a comprehensive study of the critical points of the distance function. The survey by Cheeger \cite{Che} addresses applications of the Grove-Shiohama notion of critical points. Proposition~\ref{GS2} allows us to apply this critical point theory to the study of half-geodesics. We follow established methods from critical point theory in employing the theorem of Toponogov to prove Theorem~\ref{behavior}.  

\begin{proof}[Proof of Theorem~\ref{behavior}.] We note that $\text{diam}(M) \geq d(\gamma(t),\gamma(t+\pi) ) >  \frac{\pi}{2\sqrt{H}}$ so that by the diameter sphere theorem $M$ is homeomorphic to a sphere. 

First assume by contradiction that $\sigma \colon S^1 \to M$ is a closed geodesic with $l(\sigma)<l(\gamma)$ and $\sigma(0)=p=\gamma(t)$. Then by Theorem~\ref{combined} and Proposition~\ref{GS2} we know that $p$ is a Grove-Shiohama critical point of $d_q$ where $q=\gamma(t+\pi)$. Therefore there exists a minimum geodesic $\tau$ from $p$ to $q$ with $\alpha = \measuredangle{(\dot{\sigma}(0),\dot{\tau}(0) )} \leq \pi/2$. Let $l(\sigma)= \frac{\pi}{\sqrt{H_1}} < l(\gamma) =2 \, l(\tau)= \frac{\pi}{\sqrt{H_0}}$ so that $H_0 < H_1$ and $H_0 < H$. Set $H_2 = \min{ \{H, H_1\} }$. Although $\sigma$ is not a segment we know $l(\sigma) \leq  \frac{\pi}{\sqrt{H_2}}$ so that we can apply Toponogov's theorem to the hinge $\{\sigma, \tau \}$. This yields a comparison triangle in the sphere $S_{H_2}$ with side lengths $\{ B= \frac{\pi}{2\sqrt{H_0}} , C= \frac{\pi}{\sqrt{H_1}}, D \}$ where $B \leq D \leq \frac{\pi}{\sqrt{H_2}}$. We note that $\frac{\pi}{2} < B \sqrt{H_2} \leq D \sqrt{H_2} \leq \pi$ so that by Cosine Law we have
\begin{align*}
0 &>\cos{B\sqrt{H_2}} \geq \cos{D\sqrt{H_2}} \\
&= \cos{B\sqrt{H_2}} \cdot \cos{ C\sqrt{H_2}} +  \sin{B\sqrt{H_2}} \cdot \sin{ C\sqrt{H_2}} \cdot \cos{\alpha} \\
&\geq  \cos{B\sqrt{H_2}} \cdot \cos{ C\sqrt{H_2}} 
\end{align*}
a contradiction. Thus such a closed geodesic $\sigma$ must have $l(\sigma)\geq l(\gamma)$.

Next assume by contradiction that $\sigma \colon S^1 \to M$ is a half-geodesic with  $\sigma(0)=p=\gamma(t)$ but which does not contain the point $\gamma(t+\pi)=q$. Then by Theorem~\ref{combined} and Proposition~\ref{GS2} we know that $\sigma(\pi)=x \neq q$ is a Grove-Shiohama critical point of $d_p$. Let $\gamma_2$ be a minimal geodesic from $q$ to $x$. We know there exists a minimal geodesic $\gamma_0$ from $x$ to $p$ such that $\measuredangle{(-\dot{\gamma}_2(d(q,x)),\dot{\gamma}_0(0))} \leq \pi/2$. Since $p$ and $q$ are mutually critical there exists minimal geodesics $\gamma_1$ and $\tilde{\gamma}_1$ from $p$ to $q$ such that $\measuredangle{(\dot{\gamma}_1(0),\dot{\gamma}_0(d(p,x)))} \leq \pi/2$ and $\measuredangle{(-\dot{\tilde{\gamma}}_1(d(p,q)),\dot{\gamma}_2(0))} \leq \pi/2$. Now apply the triangle version of Toponogov's theorem to both $\{\gamma_0, \gamma_1, \gamma_2 \}$ and $\{\gamma_0, \tilde{\gamma}_1, \gamma_2 \}$ yielding comparison triangles in the sphere $S_H$. Because triangles in the sphere are determined up to congruence by side lengths we get a unique comparison triangle, each of whose angles is $\leq \pi/2$. This implies that the comparison triangle is completely contained in an octant of the sphere, hence has side lengths $\leq \frac{\pi}{2\sqrt{H}}$, a contradiction to $l(\gamma_1)=l(\gamma)/2>\frac{\pi}{2\sqrt{H}}$. We conclude that the half-geodesic $\sigma$ must contain the point $\gamma(t+\pi)$ and therefore must have length $l(\gamma)$.
\end{proof}

\begin{remark}\normalfont We note that the techniques in the first part of the proof of Theorem~\ref{behavior} are inspired by Berger's proof of the minimal diameter theorem (Cf.~\cite{Pet}). The second part of the proof is similar to the Grove-Shiohama proof of the diameter sphere theorem (Cf.~\cite{Che}).
\end{remark}

\begin{ex}\normalfont \label{round} The round sphere of constant sectional curvature $H$ satisfies the assumptions of Theorem~\ref{behavior}. The only half-geodesics are the great circles which all have length $\frac{2\pi}{\sqrt{H}}$. Any half-geodesic through $p$ contains its antipode $-p$.
\end{ex}

The theorem of Bonnet-Myers states that a complete Riemannian manifold with sectional curvature $K \geq H>0$ has an upper diameter bound of $\frac{\pi}{\sqrt{H}}$ and therefore an upper bound of $\frac{2\pi}{\sqrt{H}}$ on the length of half-geodesics. Moreover, Cheng's rigidity theorem \cite[Theorem 3.1]{Cheng} states that the round sphere is the only such manifold to realize this upper diameter bound. The following example illustrates that Theorem~\ref{behavior} has applications beyond the round sphere by constructing complete Riemannian manifolds with sectional curvatures $K \geq H>0$ and admitting half-geodesics of lengths strictly between $\frac{\pi}{\sqrt{H}}$ and $\frac{2\pi}{\sqrt{H}}$. This example also demonstrates that there are non-round metrics on the sphere which admit half-geodesics of length twice the diameter.

\begin{ex}\normalfont \label{oblate} Starting with the round two-sphere of curvature $H_0 > H$ we construct an oblate ellipsoid by shrinking the sphere along the $z$-axis while keeping the equator fixed. We allow the construction to proceed until the curvature at the poles is exactly $H$. The resulting surface is a complete Riemannian manifold with sectional curvature $K \geq H>0$. In this construction the meridians remain half-geodesics and have length twice the new diameter. We note that the equator now fails to minimize between antipodal points and is no longer a half-geodesic. If $H_0$ is not too much greater than $H$ then the new meridians will have length strictly between $\frac{\pi}{\sqrt{H}}$ and $\frac{2\pi}{\sqrt{H}}$. The conclusions of Theorem~\ref{behavior} therefore apply to any closed geodesics that intersect a meridian.
\end{ex}

The next example demonstrates that the strict lower bound imposed on the length of the half-geodesics in Theorem~\ref{behavior} is sharp.

\begin{ex}\normalfont \label{proj} The standard real projective space of constant sectional curvature $H$ admits half-geodesics of length $\frac{\pi}{\sqrt{H}}$. Through any point $[p]$ on the identified equator there exists a half geodesic $\gamma$ with $\gamma(0)=[n]$ and $\gamma(\pi)=[p]$. Thus there exist multiple half-geodesics through $[n]$ which fail to meet again after time $\pi$. 
\end{ex}

\begin{cor}\label{half_geo_cor} Let $M$ be a complete Riemannian manifold with sectional curvature $K \geq H >0$ and diameter $> \frac{\pi}{2\sqrt{H}}$. Let  $p,q \in M$ be points which realize the diameter. Then any half-geodesic through $p$ must contain $q$ and have length twice the diameter.
\end{cor} 
\begin{proof} Berger's lemma (Cf.~\cite{CE}) says that diameter realizing points are always mutual Grove-Shiohama critical points. The proof of Theorem~\ref{behavior} shows that $q$ is the unique Grove-Shiohama critical point of $d_p$ and the corollary follows.
\end{proof}

\begin{lemma}[Gromov, Cf. \cite{Che}]\label{gromov2} Let $M$ be a complete non-compact Riemannian manifold with sectional curvature $K \geq 0$. Then for every $p \in M$ the distance function $d_p$ has no critical points outside of some ball $B(p,R_p)$. In particular, $M$ is homeomorphic to the interior of a compact manifold with boundary. 
\end{lemma}

We note that in the compact setting we always have a universal upper bound on the length of the half-geodesics, namely twice the diameter. In the non-compact setting such a universal bound does not exist, but we offer the following pointwise result.

\begin{cor} Let $M$ be a complete non-compact Riemannian manifold with sectional curvature $K \geq 0$. Then for every $p \in M$ there exists a constant $R_p$ such that any half-geodesic through $p$ must have length $\leq 2 \, R_p$. 
\end{cor}
\begin{proof} This follows from Proposition~\ref{GS2} and Lemma~\ref{gromov2}.
\end{proof}

As a final application of Theorem~\ref{combined} and Proposition~\ref{GS2} we give a positive result on the existence of half-geodesics on the flat two-torus. 

\begin{ex}[half-geodesics on 2-torus]\label{torus}\normalfont We use our methods to show that any flat torus $T$ with fundamental domain a rectangle has exactly four half-geodesics through any point $p \in T$. Let $p$ be at the center of a fundamental domain for $T$. It is known that the midpoints of the sides together with the corners of the domain project to the three Grove-Shiohama critical points of $d_p$ \cite[Example 1.7]{Che}. By Proposition~\ref{GS2} we identify four candidate half-geodesics through $p$: the two lines through $p$ parallel to the sides and the two diagonals of the domain. By symmetry of the torus we have that each of these candidate half-geodesics corresponds to a class of balanced points of the uniform energy in $T^2$. We therefore apply Theorem~\ref{combined} to conclude that each of these lines is indeed a half-geodesic.
\end{ex}

\section{Acknowledgements} 
The author would like to thank Carolyn Gordon and Craig Sutton for their guidance through all stages of the research process. The author would also like to thank Christina Sormani for suggesting the original problem and for helpful discussions and direction.


\nocite{*}
\bibliographystyle{abbrvnat}
\bibliography{critical}

\end{document}